\documentclass[a4paper,12pt]{article}
\usepackage[T2A]{fontenc}                      
\usepackage[cp1251]{inputenc}           
\usepackage[all]{xy}
\usepackage{amssymb}
\usepackage{cite}
\usepackage{cmap}
\usepackage{latexsym}
\usepackage{enumerate}
\usepackage{amsmath, amsthm, amscd, amsfonts, amssymb, graphicx, color}
\usepackage[left=2.5cm,right=2.5cm,top=2cm,bottom=2cm]{geometry}
\usepackage{indentfirst}
\usepackage{array}

\usepackage{float}

\usepackage{wrapfig}

\newtheorem{corollary}{Corollary}
\newtheorem{conjecture}{Conjecture}

\newtheorem{remark}{Remark}

\newtheorem{theorem}{Theorem}
\newtheorem{lemma}{Lemma}
\newtheorem{definition}{Definition}
\newtheorem{assertion}{Assertion}
\newtheorem{problem}{Problem}

\begin{document}

\title{Proof of the strong conjecture about $F$-irregular graphs in the class of graphs $\{F\}$ of diameter $2$}
\date{}
\author
{Tatiana Dovzhenok\thanks{Research Laboratory "Mathematics of Hybrid Intelligence Systems", 	Francisk Skorina Gomel State University, 246028,  Gomel,  Belarus. 		
		E-mail: \texttt{t.dovzhenok@mail.ru}}}
\maketitle

\begin{abstract}	
	Let $F$ and $G$ be simple finite undirected graphs. A graph $G$ is called $F$-irregular if any two of its distinct vertices belong to different numbers of copies of $F$ in $G$. According to the strong conjecture about $F$-irregular graphs (Dovzhenok, Filuta, Chuhai), for any connected graph $F$ of order $|F|\geqslant 3$, there exist infinitely many $F$-irregular graphs. In the present paper, the strong conjecture about $F$-irregular graphs is confirmed in the class of graphs $\{F\}$ of diameter $2$. It is proved that for every graph $F$ of diameter $2$, there exists an infinite series of $F$-irregular graphs of diameter $3$.\
	
\end{abstract}

\section{Introduction and statement of the problem}
\label{sec1} 

All graphs considered in this paper are assumed to be non-trivial, simple, finite, and undirected.

We begin with a well-known statement whose proof can be found, for example, in \cite{r1}.

\begin{assertion} \label{as1} 
	There is no graph in which the degrees of all vertices are pairwise distinct. 		
\end{assertion}

Such non-existent graphs were called \emph{perfect} by Behzad and Chartrand   \cite{r1}, and later renamed \emph{irregular}.
Thus, an irregular graph was originally positioned as a complete opposite to a  \emph{regular} graph, where all vertices have the same degree. Over time, the term "irregular graph" came to be applied to graphs whose vertices differ according to some attribute. For instance, in \cite{r2}, the highly irregular graphs were introduced and studied: for each vertex in such graphs, all its neighbors have different degrees.  In \cite{r3}, an alternative approach to defining irregular graphs is considered, shifting focus from the vertices themselves to their neighborhoods. A graph is called link-irregular if the subgraphs induced by the neighborhoods of its vertices are pairwise non-isomorphic.

For a more detailed acquaintance with various concepts of irregular graphs, we recommend the excellent book "Irregularity in graphs" \cite {gtwa}. In addition to well-known results, it features a list of open problems and a number of interesting conjectures on the topic of irregular graphs.  One such problem pertains to the underexplored domain of $F$-\emph{regular graphs}.

In 1987, Chartrand, Holbert, Oellermann, and Swart \cite{r4} introduced a new graph-theoretic concept, the $F$-\emph{degree of a vertex}, which allowed us to talk about the uniqueness of graph vertices in terms of their $F$-degrees and led to the emergence of the theory of $F$-irregular graphs.

\begin{definition}\label{def1}  
	Let $F$ and $G$ be graphs. The $F$-degree of a vertex $v$ in the graph $G$ is defined as  the number of subgraphs of $G$ that are isomorphic to $F$ and contain $v$.
\end{definition}
\begin{definition}\label{def2}  
	A graph $G$ is called $F$-irregular if any two distinct vertices in $G$ have different $F$-degrees.
\end{definition}

Expressed in terms of  $F$-irregular graphs, Assertion \ref{as1} is equivalent to the following.

\begin{assertion} \label{as2}  
	{Let $K_2$ be a complete graph on two vertices. No graph is $K_2$-irregular.}  		
\end{assertion}

A natural fundamental question arises when $F \neq K_2$. 

\begin{problem} \label{p1}
	Does there exist an $F$-irregular graph for every connected graph $F \neq K_2$? 
\end{problem}

In \cite{r4}, Problem \ref{p1} was studied for two classes of graphs -- complete graphs and stars.

\begin{theorem} \label{th1} 
	Let $K_{n}$ and $K_{1,n-1}$ be a complete graph and a star of order $n \geqslant 3$. Then, there exist graphs that are $K_{n}$-irregular  and  $K_{1,n-1}$-irregular, respectively.
\end{theorem} 

The same paper \cite{r4} formulated a central conjecture concerning  $F$-irregular graphs.

\begin{conjecture} \label{con1}
	For every connected graph $F$ on three or more vertices, there exists an $F$-irregular graph.
\end{conjecture}

Unfortunately, the concept of $F$-irregular graphs, despite its elegance and beauty, has not gained widespread popularity. 
This is explained primarily by the difficulty of constructing $F$-irregular graphs for an arbitrary graph $F$, as well as the lack of a sufficiently developed mathematical apparatus for comparing $F$-degrees of vertices. For a long time, no other infinite family of graphs, apart from $K_{n}$ and $K_{1,n-1}$, was found that would confirm Conjecture \ref{con1}.

In 2024, Dovzhenok, Filuta, and Chuhai \cite{r5} presented a new result on the $F$-irregular graph problem.

\begin{theorem} \label{th2}  
	For any biconnected graph $F$ whose minimum vertex degree is $2$, there exist infinitely many $F$-irregular graphs.
\end{theorem}

This result, together with the question of the existence of $F$-irregular graphs, raises the issue of their quantity.

\begin{problem} \label{p2}
	For every connected graph $F \ne K_2$, determine whether the set of all $F$-irregular graphs finite or infinite?
\end{problem}

In  \cite{r5}, the following assumption was put forward, which strengthens the statement of Conjecture \ref{con1}.

\begin{conjecture} \label{con2} \emph{(Strong conjecture about  $F$-irregular graphs)}
	For any connected graph $F$ of order $|F|\geqslant 3$, there are infinitely many $F$-irregular graphs.
\end{conjecture}

Conjecture \ref{con2} holds for the path $P_3$ on three vertices. This fact is a direct consequence of the following theorem, proved by Salehi in  \cite{r6}.

\begin{theorem} \label{th3} 
	For every integer $n \geqslant 6$, there exists at least one graph of order $n$ with distinct $P_3$-degrees of vertices.
\end{theorem}

Currently, biconnected graphs with minimum degree $2$, along with the path $P_3$, exhaust the set of known graphs for which Conjecture \ref{con2} holds. Therefore, finding new infinite series of graphs confirming Conjecture \ref{con2}, as well as developing methods for comparing $F$-degrees of vertices, remain a significant area of interest within the field of $F$-irregular graphs.

The purpose of this study is to confirm the strong conjecture about $F$-irregular graphs for any graph $F$ of diameter $2$.

In section \ref{sec2} of our  work, for an arbitrary graph $F$ of diameter $2$, we construct a family of graphs $\{F_{2l}\}$ on $2l$ vertices that will be $F$-irregular under certain restrictions on the parameter $l$.  We introduce  new methods for comparing the $F$-degrees of vertices in such graphs.
In particular, when the minimum degree of $F$ is $1$, the features of the constructions allow us to compare the $F$-degrees of some vertices $v_1$ and $v_2$ in $F_{2l}$ by establishing an injective, but not surjective, mapping between the sets of subgraphs of $F_{2l}$  isomorphic to $F$ and containing the vertices $v_1$ and $v_2$, respectively. 
\vspace{3pt}

The main result of this paper is Theorem \ref{th4}, whose proof is presented in section \ref{sec3}.

\begin{theorem} \label{th4} 
	For each graph $F$ of diameter $2$, there exist infinitely many $F$-irregular graphs of diameter $3$.
\end{theorem}

In the final part of the study, we define a new class of graphs,  $n$-\emph{hyper-irregular graphs}, and formulate a more general assumption than  Conjecture \ref{con2} -- the super-strong conjecture about $F$-irregular graphs.
\vspace{3pt}

Notations utilized in the work are as follows:  
\vspace{3pt}

$\bullet$ 
$V(G)$ and $E(G)$ are the set of vertices and the set of edges of graph $G$, respectively;

\vspace{3pt}
$\bullet$
$(u, v)$ is the graph edge incident to vertices $u, v$, with  $u \ne v$;

\vspace{3pt}
$\bullet$
$deg_G (v)$ and $Fdeg_G (v)$ are the degree and $F$-degree of vertex $v$ in graph $G$, respectively;

\vspace{3pt}
$\bullet$
$\delta (G)$ is the minimum of all vertex degrees in graph $G$;

\vspace{3pt}
$\bullet$
$|M|$ is the number of all elements in set $M$;

\vspace{3pt}
$\bullet$
$|G|=|V(G)|$ is the order of graph $G$;

\vspace{3pt}
$\bullet$
$N_G(v)$ is the neighborhood of vertex $v$ in graph $G$, i.e., the set of all vertices of $G$ adjacent to $v$;

\vspace{3pt}
$\bullet$
$N_G[v]=N_G(v)\cup\{v\}$ is the closed neighborhood of vertex $v$ in graph $G$;

\vspace{4pt}
$\bullet$
$C_n ^k$ is the number of $k$-element subsets of an $n$-element set, that is, for integers 
\vspace{3pt}

\hspace{-21pt} $n \ge k \ge 0$, $C_n ^k= \frac {n!}{k!(n-k)!}$, where $0!= 1$, $m!= 1 \cdot 2 \cdot 3 \cdot ... \cdot m$ if $m$ is a positive integer, and $C _n ^k=0$ otherwise.

\section{Basic graph constructions} \label{sec2} 

\begin{assertion} \label{as3} 
	{For every graph $F$ of diameter $2$, the inequality   
		$\delta (F) \leqslant |F|-2$ holds.}		
\end{assertion}

\begin{definition}\label{def3}  
	Consider an arbitrary graph $F$ of diameter $2$. Let $n=|F|$, $t= \delta (F)$.	
\end{definition}	

Further, throughout section \ref{sec2}, we assume that the graph $F$ of diameter $2$ is fixed, together with its characteristics $n, t$, where $t \leqslant n-2$ as per Assertion \ref{as3}.

\subsection{Graph $A_{2l-1}$} 

\begin{definition}\label{def4}  
	For each integer $l>n$, consider the graph $A_{2l-1}$ of the following type:	
		$$V(A_{2l-1} )=\{1,2,...,2l-1\},
		E(A_{2l-1})=\{(i,j)|i,j \in V(A_{2l-1}), i\ne j, |i-j| \le l-1\}.$$	
\end{definition}

For clarity, let us depict the vertices of the graph $A_{2l-1}$ in the form of two levels: vertices $1,2,...,l$ are placed on the top level in increasing order from left to right, and vertices $l+1,l+2,...,(2l-1)$ on the bottom level in such a way that any vertex $i$ on the bottom level is positioned directly under vertex $i-l+1$ on the top level (see Figure 1).

\begin{figure}[ht]
	\centering
	\includegraphics[width=12cm]{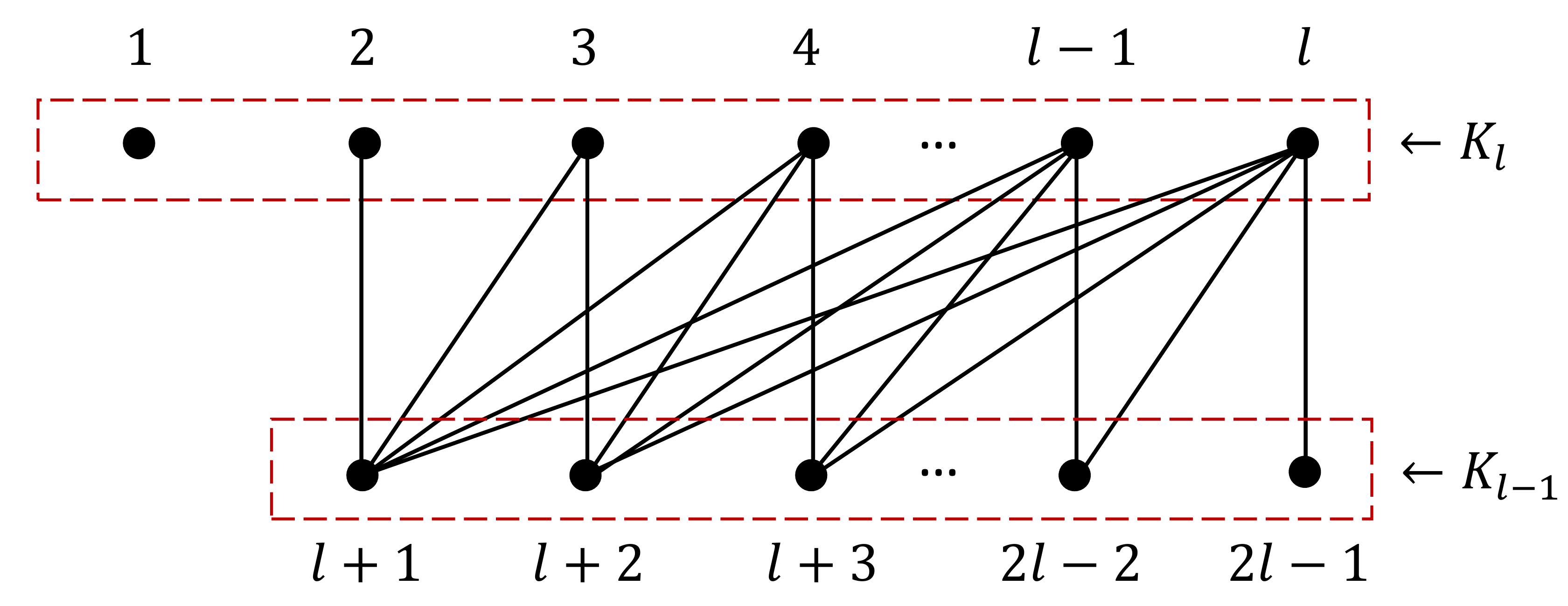}
	\caption{Graph $A_{2l-1}$.}
	\label{fig1}
\end{figure}

Then any two distinct vertices on each level of  $A_{2l-1}$ will be adjacent, and any vertex of the bottom level will be adjacent to all vertices of the top level located directly above it or to the right.

\begin{definition}\label{def5}  
	Let $i\in V(A_{2l-1} )$. Denote by $z_i$ the $F$-degree of vertex $i$ in the graph   $A_{2l-1}$.	
\end{definition}

\begin{lemma} \label{lem1} 
	For the $F$-degrees of the vertices on the top and bottom levels of $A_{2l-1}$, the following equalities hold:
	$$z_i=z_{2l-i} \quad \forall i \in V(A_{2l-1}).$$	
\end{lemma}

\begin{proof} 
Consider the mapping
$f:V(A_{2l-1}) \rightarrow V(A_{2l-1})$ such that
$$f(v)=2l-v \quad \forall v \in V(A_{2l-1}).$$
The mapping $f$ is clearly bijective. Moreover, the following equalities are true:
$$|f(i)-f(j)|=|i-j| \quad \forall i,j \in V(A_{2l-1}).$$
Consequently, for any $i,j \in V(A_{2l-1}), i\ne j$, the inequality $|f(i)-f(j)| \leqslant l-1$ is equivalent to the inequality $|i-j| \leqslant l-1$, and so vertices  $f(i)$ and $f(j)$ are adjacent in the graph $A_{2l-1}$ if and only if vertices $i$ and $j$ are adjacent in  $A_{2l-1}$. 
Thus, the mapping $f$ preserves the adjacency property of vertices in $A_{2l-1}$ and, given the bijectivity of $f$, is an automorphism of the graph $A_{2l-1}$.  
Hence, for each $i \in V(A_{2l-1})$, the vertices 
$i$ and $f(i)=2l-i$ have the same $F$-degrees in $A_{2l-1}$, that is, $z_i=z_{2l-i}$.  
\end{proof}

\begin{lemma} \label{lem2} 
	The $F$-degrees of the top-level vertices in the graph $A_{2l-1}$ satisfy the following inequalities:
	$$z_{i+1} \geqslant z_i+C_{l-2}^{n-2} \quad \forall i \in \{1,2,...,l-1\}.$$		
\end{lemma}

\begin{proof} 
	Fix $i\in \{1,2,...,l-1\}$.
	
	Let us construct a lower bound for the difference $z_{i+1}-z_i$ of the $F$-degrees of vertices $i+1$ and $i$ in the graph $A_{2l-1}$.
	To do this, we consider the graph $H=A_{2l-1}\backslash (i+1,l+i)$ (see Figure 2).
	
	\begin{figure}[ht]
		\centering
		\includegraphics[width=12cm]{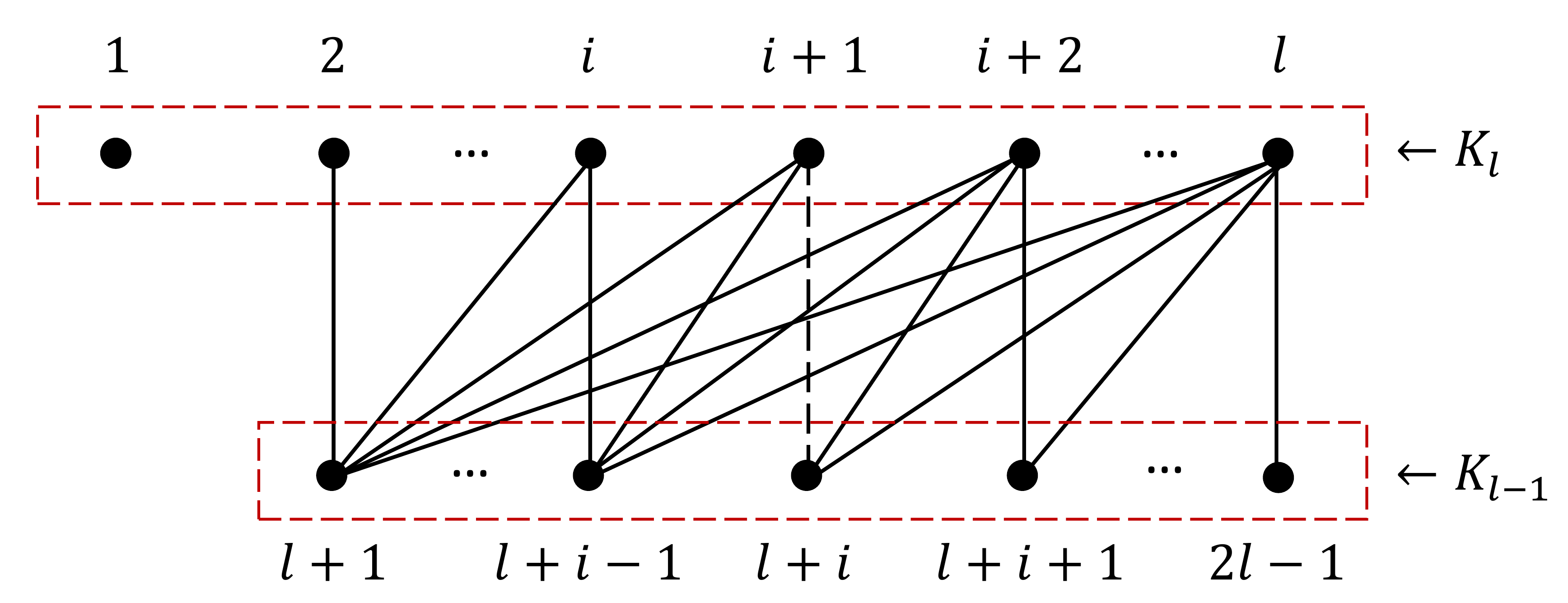}
		\caption{Graph $H=A_{2l-1}\backslash (i+1,l+i)$.}
		\label{fig2}
	\end{figure}
	
	Note that the closed neighborhoods of vertices $i$ and $i+1$ in the graph $H$ coincide:
	$$N_H[i]=N_H[i+1] =\{1,2,...,l+i-1\}.$$
	This implies that there exists an automorphism $g:V(H)\rightarrow V(H)$ such that
	$$g(i)=i+1,g(i+1)=i, g(v)=v \quad \forall v \in V(H) \backslash \{i, i+1\}.$$
	Therefore, the $F$-degrees of vertices $i$ and $i+1$ in the graph $H$ are equal, that is,
	$$Fdeg_H (i+1)=Fdeg_H (i).$$
	
Let us return to the graph $A_{2l-1}$. Then, taking into account the above, we have
$$z_{i+1}=Fdeg_H (i+1)+x=Fdeg_H (i)+x, \quad z_i=Fdeg_H (i)+y,$$
where $x$ is the number of subgraphs of $A_{2l-1}$, each of which is isomorphic to $F$ and contains the edge $(i+1,l+i)$, and $y$ is the number of subgraphs of $A_{2l-1}$  isomorphic to $F$, containing the edge $(i+1,l+i)$ and the vertex $i$. Therefore, the difference between the $F$-degrees of vertices $i+1$ and $i$ in the graph $A_{2l-1}$ is equal to  $$z_{i+1}-z_i=|M|,$$
where $M$ is the set of all subgraphs of $A_{2l-1}$ that are isomorphic to $F$, contain the edge $(i+1,l+i)$, and do not contain the vertex $i$.

We prove that $$|M|\geqslant C_{l-2}^{n-2}.$$		
For this purpose, we consider only those subgraphs of $M$ (let's call them special) whose vertices belong to the set $\{i+1,i+2,...,l+i\}$.
Note that any $2$ vertices from the set $\{i+1,i+2,...,l+i\}$ are adjacent in  $A_{2l-1}$.
Therefore, by adding to the edge $(i+1,l+i)$ any $n-2$ distinct vertices from the $(l-2)$-element set $\{i+2,i +3,...,l+i-1\}$  and the required set of edges, it is always possible to form at least one special subgraph. Moreover, the special subgraphs formed based on different choices of $(n-2)$ vertices are clearly distinct.  Consequently, the graph $A_{2l-1}$ contains at least  $C_{l-2}^{n-2}$ special subgraphs. Hence 
$$z_{i+1} - z_i=|M| \geqslant C_{l-2}^{n-2},$$
from where we obtain the required inequality $z_{i+1} \geqslant z_i+C_{l-2}^{n-2}$. 	
\end{proof}

\begin{corollary} \label{cor1} 
	The $F$-degrees of vertices on the top level of  $A_{2l-1}$ monotonically increase  with the vertex number:
	$$z_1<z_2<z_3< \cdots <z_{l-1}<z_l.$$		
\end{corollary}	

\subsection{Graph  $F_{2l}$} 

\begin{definition}\label{def6}  
	For each integer $l>n$, define the graph $F_{2l}$ as follows
	 (graph $F_{2l}$ is shown in Figure 3): 
	
	\begin{center}
		$V(F_{2l})=V(A_{2l-1})\cup\{2l\}$,
		$E(F_{2l})= E(A_{2l-1})\cup \{(1,2l),(2,2l),..., (t,2l)\}$.
	\end{center}  
\end{definition}

\begin{figure}[ht]
	\centering
	\includegraphics[width=12cm]{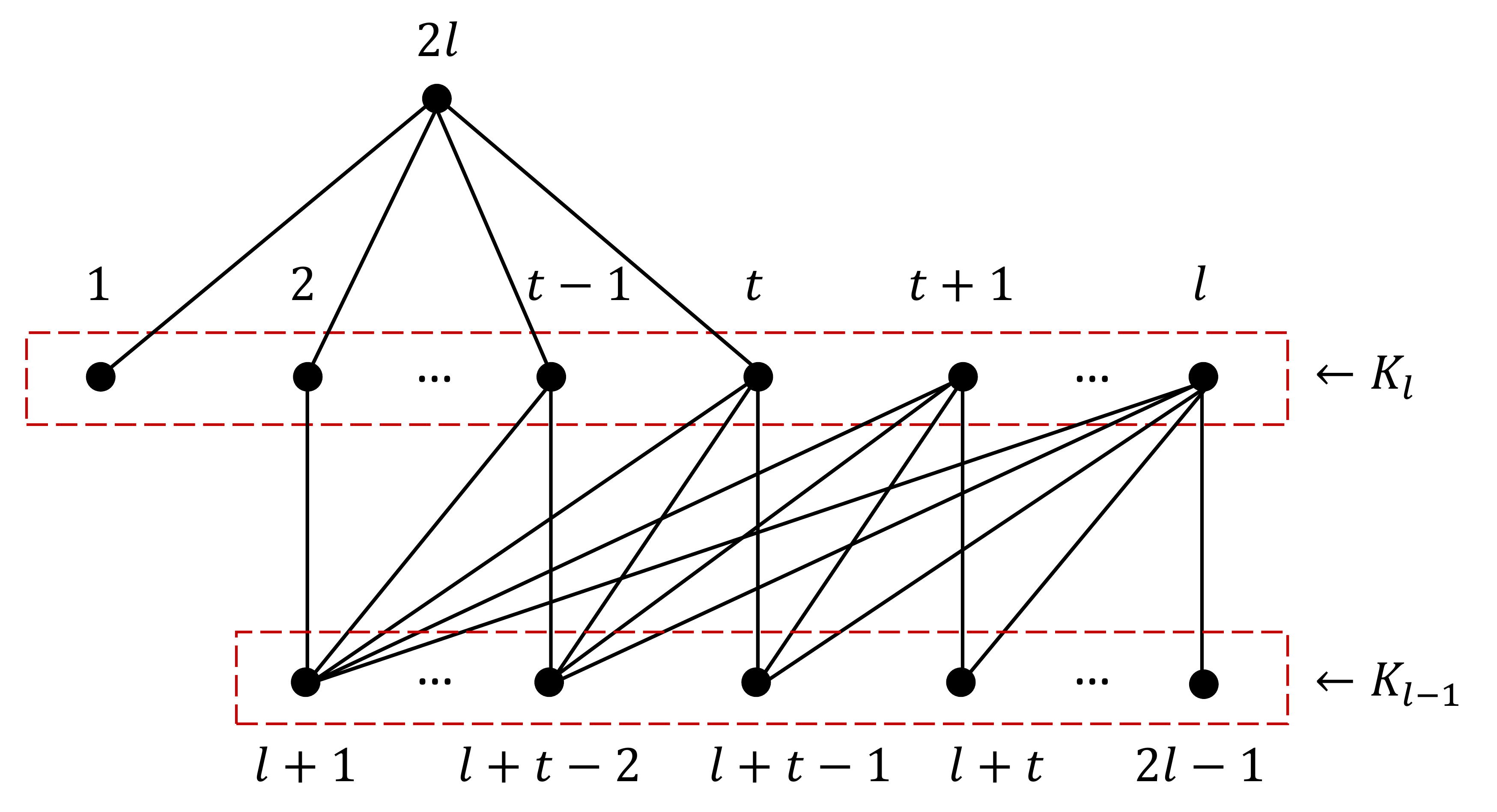}
	\caption{Graph $F_{2l}$.}
	\label{fig3}
\end{figure}

\begin{lemma} \label{lem3} 
	The diameter of the graph $F_{2l}$ is equal to $3$.
\end{lemma}

\begin{proof} 
	First of all, we note that every vertex of  $F_{2l}$ different from $l$ and $2l$ is adjacent to the vertex $l$, while the vertex $2l$ is adjacent to the vertex $1$. Therefore, the distance between any two distinct vertices in  $F_{2l}$ does not exceed $3$. On the other hand, any vertex $v \in \{l+t,l+t+1,...,2l-1\}$ lies at distance  $3$ from the vertex $2l$, since $v$, due to the inequality $l>n\geqslant t+2$, does not belong to the neighborhood $N_{F_{2l}}(2l)=\{1,2,...,t\}$ of the vertex $2l$ in  $F_{2l}$, and is also not adjacent to any vertex from the set $N_{F_{2l}}(2l)$. Consequently, the graph $F_{2l}$ has diameter $3$.		
\end{proof}

\begin{definition}\label{def7}  
	Let $i\in V(F_{2l})$.  Denote by $f_i$ the $F$-degree of vertex $i$ in the graph $F_{2l}$. 	
\end{definition}

\begin{definition}\label{def8}  
	Denote by $L$ the set of all subgraphs of $F_{2l}$ that are isomorphic to $F$ and contain the vertex $2l$. 
\end{definition}

\begin{definition}\label{def9}  
	Let $X$ denote the graph induced by the  vertex set $\{1,2,...,l+t-1\}\cup\{2l\}$ of the graph $F_{2l}$.
\end{definition}

\begin{lemma} \label{lem4} 
	For any subgraph $Y\in L$, the following holds:
	
	$1)$ $(1,2l),(2,2l), ..., (t,2l) \in E(Y)$;
	
	$2)$ $Y \subseteq X$.
\end{lemma}
\begin{proof} 
	Let $Y\in L$.
	
	We prove $1)$.
	Obviously, $\delta (Y)=t$, since the subgraph $Y$ is isomorphic to the graph $F$, and $\delta (F)=t$. Consequently, the vertex $2l \in V(Y)$ is adjacent to at least \(t\) distinct vertices in $Y$. Further, from the equalities $deg_{F_{2l}}(2l)=t$, $N_{F_{2l}}(2l)=\{1,2,...,t\}$, and the fact that $Y$ is a subgraph of $F_{2l}$, we obtain the desired result.
 	
	We prove $2)$. Note that the distance from the vertex $2l\in V(Y)$ to any other vertex of \(Y\) does not exceed $2$, since $Y$ is isomorphic to the graph $F$ with diameter $2$. On the other hand, in the graph $F_{2l}$, any vertex from the set $\{l+t,l+t+1,...,2l-1\}$ is located at distance $3$ from the vertex $2l$ according to the proof of Lemma \ref{lem3}. Therefore, no vertex from this set belongs to  $Y$, which is a subgraph of $F_{2l}$. Hence, $Y \subseteq X$.
\end{proof}

\begin{definition}\label{def10}  
	Let $i \in V(A_{2l-1})$. The quantity $\delta_i=f_i-z_i$ will be called the jump of vertex $i$ when  transitioning  from the graph $A_{2l-1}$ to the graph $F_{2l}$.	
\end{definition} 

\begin{assertion} \label{as4} 
	{The jump of vertex $i$, where $i\in\{1,2,...,2l-1\}$, is equal to the number of subgraphs from the set $L$ that contain  vertex  $i$.} 		
\end{assertion}

\begin{lemma} \label{lem5} 
	The following equalities hold for the jumps of vertices in the graph $A_{2l-1}$:
	
	$ 1) \: \delta_i=f_{2l} \quad \forall i \in \{1,2,...,t\} $,
	
	$ 2) \: \delta_i = \delta_l \quad \forall i \in \{t+1,t+2,...,l\} $,
	
	$ 3) \: \delta_i=0 \quad \forall i \in \{l+t,l+t+1, ...,2l-1\}$.	
\end{lemma}
 
\begin{proof} 
	The truth of equalities $1)$ and $3)$ follows directly from Lemma \ref{lem4} and Assertion \ref{as4}.

	We prove $2)$. Fix $i\in \{t+1,t+2,...,l\}$. According to Lemma \ref{lem4} and Assertion \ref{as4}, the jump of vertex $i$ is equal to the number of all subgraphs of $X$ that are isomorphic to  $F$ and contain vertices $2l$ and  $i$.
	Furthermore, since the closed neighborhoods of vertices $i$ and $l$ in the graph $X$ coincide
	$$N_X[i]=N_X[l] =\{1,2,...,l+t-1\},$$
	then there exists an automorphism $f_X:V(X)\rightarrow V(X)$ such that $$f_X(i)=l,f_X(l)=i, f_X(v)=v \quad \forall v \in V(X) \backslash \{i, l\}.$$
	Consequently, vertices $i$ and $l$ belong to the same number of subgraphs of $X$ that are isomorphic to $F$ and contain vertex $2l$. Therefore, the equality $\delta_i = \delta_l$ holds for each $i \in\{t+1,t+2,...,l\}$.		
\end{proof}	

\begin{lemma} \label{lem6} 
	The $F$-degree of vertex $2l$ in the graph $F_{2l}$ satisfies the inequality
	$$ f_{2l} \leqslant n!C_{l-1}^{n-t-1}.$$	
\end{lemma}

\begin{proof} 
	We construct an upper bound for $f_{2l}=|L|$. Note that every subgraph of $L$ is isomorphic to $F$, and hence, has order $n$. Moreover, by Lemma \ref{lem4}, each subgraph of $L$ contains vertices $1,2,...,t,2l$, while the remaining $n-t-1$ vertices of such subgraphs (distinct from $1,2,...,t,2l$)  belong to the set $\{t+1,t+2,...,l+t-1\}$ of cardinality $l-1$. Therefore, there are exactly $C_{l-1}^{n-t-1}$ ways to choose the $n-t-1$ remaining vertices, and so there are at most $C_{l-1}^{n-t-1}$ different sets of $n$ vertices, each of which is the vertex set of some subgraph of $L$. Furthermore, since any graph of order $n$ contains at most $n!$ subgraphs isomorphic to the graph $F$ of order $n$, then for each of these sets  there are at most $n!$ subgraphs from $L$ with that vertex set. Thus, we obtain the estimate:
	$$f_{2l}=|L| \leqslant n!C_{l-1}^{n-t-1}.$$	
\end{proof}

\begin{lemma} \label{lem7} 
	For the jump of vertex $l$ in the graph $A_{2l-1}$, the following inequality holds:
	$$\delta_l \leqslant n!C_{l-2}^{n-t-2}.$$	
\end{lemma}

\begin{proof} 
	Let us construct an upper bound for $\delta_l$. By Assertion \ref{as4}, the jump $\delta_l$ coincides with the number of all subgraphs of $L$ containing vertex $l$. Taking into account Lemma \ref{lem4}, each such subgraph contains $t+2$ vertices $1,2,...,t,2l,l$, and the remaining $n-t-2$ vertices belong to the set $\{t+1,t+2,...,l+t-1\}\setminus \{l\}$ of cardinality $l-2$. It follows that the $n-t-2$ remaining vertices of such subgraphs can be chosen in exactly \(C_{l-2}^{n-t-2}\) ways. Next, similarly to the proof of Lemma \ref{lem6}, we establish the truth of the inequality $\delta_l \leqslant n!C_{l-2}^{n-t-2} $.	 
\end{proof}

\begin{lemma} \label{lem8}  
	If $C_{l-1}^{n-1}> n!C_{l-1}^{n-t-1}$, then the $F$-degree of  vertex $2l$ in the graph $F_{2l}$ satisfies the following estimate:  $$f_{2l}<z_1.$$	
\end{lemma}

\begin{proof} 
	We construct a lower bound for $z_1$, the number of all subgraphs of the graph $A_{2l-1}$ that are isomorphic to $F$ and contain the vertex $1$.
	First, note that any two distinct vertices in the set $\{1, 2, \ldots, l\}$ are adjacent in $A_{2l-1}$. Therefore, from any subset of $n-1$ distinct vertices in $\{2,3,\ldots,l\}$ together with vertex $1$, by adding the necessary edges, one can always form at least one subgraph of $A_{2l-1}$ isomorphic to $F$.	
	It is easy to see that different $(n-1)$-element subsets of $\{2,3,\ldots,l\}$, together with vertex $1$, correspond to distinct subgraphs of $A_{2l-1}$ isomorphic to $F$. Therefore, $z_1$ 
	is at least the number of unordered samples of size $n-1$ chosen from the $(l-1)$-element set $\{2,3,...,l\}$, that is, 
	$$z_1 \geqslant C_{l-1}^{n-1}.$$
		
	Next, from the inequalities $f_{2l} \leqslant n!C_{l-1}^{n-t-1} $ (Lemma \ref{lem6}) and $C_{l-1}^{n-1}> n!C_{l-1}^{n-t-1}$ (the assumption of Lemma \ref{lem8}), we deduce that $f_{2l} < z_1$.		
\end{proof}	

\begin{lemma} \label{lem9} 
	If $t=1$ or $t \geqslant 2$, $C_{l-2}^{n-2}>n!C_{l-1}^{n-t-1}$, then the $F$-degree of vertex $2l$ in the graph $F_{2l}$ satisfies the inequalities
	$$0 < f_{2l} < z_{i+1}-z_i \quad \forall i \in \{1,2,...,t\}.$$ 			
\end{lemma}

\begin{proof} 
	Obviously, $f_{2l}>0$, since the graph $F_{2l}$ contains at least one subgraph, for example, with the vertex set $1,2,...,n-1,2l$, isomorphic to  $F$ and containing the vertex $2l$.
	
	Next, fix a vertex $i\in \{1,2,...,t\}$ and prove that $f_{2l} <z_{i+1}-z_i$. Consider $2$ cases.
	
	\emph{Case 1}. $t \geqslant 2$.
	
	In Lemmas \ref{lem6} and \ref{lem2}, the following inequalities were proved:
	$$f_{2l} \leqslant n!C_{l-1}^{n-t-1}, \quad z_{i+1}-z_i \geqslant C_{l-2}^{n-2} \quad \forall i \in \{1,2,...,t\}.$$
	And since $C_{l-2}^{n-2}>n!C_{l-1}^{n-t-1}$ by virtue of the assumption of Lemma \ref{lem9}, then the inequality $f_{2l} <z_{i+1}-z_i$ holds for each $i\in \{1,2,...,t\}$.	
	\vspace{3pt}
	
	\emph{Case} $2$. $t=1$.
	
	In this case, we must show that the inequality $f_{2l} < z_2-z_1$ is true.  
	From the proof of Lemma \ref{lem2} it follows that
	$$z_2-z_1=|M|,$$
	where $M$, taking into account the structure of the graph $F_{2l}$, is the set of all subgraphs of $F_{2l}$ that are isomorphic to $F$, contain the edge $(2,l+1)$, and do not contain the vertices $1$, $2l$. Hence,   
	$$f_{2l} < z_2-z_1  \iff |L|<|M|.$$ 
	
	To prove the last inequality, we first note that, by Lemma \ref{lem4}, any subgraph of $L$ contains  the edge $(1,2l)$ and all its vertices belong to the set $\{1,2,...,l\}\cup\{2l\}$.
		
	Next, consider the mapping $h:L\rightarrow M$, which maps each subgraph $H\in L$ to a subgraph of $M$ according to the following rule:
	
	\hspace{100pt}	$2l \rightarrow 2, \quad 1 \rightarrow l+1, \quad 2 \rightarrow l$, if $2 \in V(H) $,
		
	\hspace{100pt} $i \rightarrow i+l-1 \quad \forall i \in V(H)\backslash \{1, 2, 2l\} $,
		
	\hspace{100pt} $(1,2l) \rightarrow (l+1,2)$, \quad $(1,2)\in E(H) \rightarrow (l+1,l)$,
		
	\hspace{100pt} $(1,i) \in E(H) \rightarrow (l+1, i+l-1) $, if $i \notin \{2,2l\}$,
		
	\hspace{100pt} $(2,i) \in E(H) \rightarrow (l,i+l-1)$, if $i \neq 1$, 
		
	\hspace{100pt} $(i,j) \in E(H) \rightarrow (i+l-1, j+l-1)$, if $i,j \notin \{1,2\}$. 
		
	It is straightforward to verify that the mapping $h:L\rightarrow M$ is injective. 		
	Now we prove that $h$ is not surjective. Indeed, the image of any subgraph of $L$ under $h$ is a subgraph of $M$ with a pendant (of degree $1$) vertex $2$. 
	However, $M$ contains a subgraph $M_1$ 
	such that 
	$$V(M_1)=\{2,3,...,n\} \cup \{l+1\}, \,  (2,l+1) \in E(M_1),$$
	$$deg_{M_1} (l+1)=1, \quad deg_{M_1} 2>1.$$			
	The existence of the specified subgraph $M_1$ follows from the facts that $(2,l+1) \in E(F_{2l})$, $\delta (F)=1$, and any distinct vertices from the set $\{2,3,...,n\}$ are adjacent in  $F_{2l}$, as well as from the connectedness of $F$ and the inequality $|F|>2$. Thus, in the subgraph $M_1$, isomorphic to $F$, the vertex $2$, which is adjacent to the pendant vertex $l+1$, is not pendant itself.
	Therefore, $M_1$ 	
	has no preimage in $L$ under $h$. 
	Then the mapping $h:L\rightarrow M$ is injective but not surjective. And since the sets $L$ and $M$ are finite, we obtain the desired inequality $|L|<|M|$.
\end{proof}

\begin{lemma} \label{lem10}  
	Let $C_{l-2}^{n-2}>n!C_{l-2}^{n-t-2}$. For the jump of vertex $l$ in the graph $A_{2l-1}$, the following inequalities hold:
	$$ 0 < \delta_l < z_{i+1}-z_i \quad \forall i \in \{t+1,t+2,...,l-1\}.$$		
\end{lemma}

\begin{proof} 
	We first prove that $\delta_l>0$.
		
	According to Assertion \ref{as4}, it suffices to show that there exists a subgraph from the set $L$ containing vertex $l$. An example of such a subgraph is $L_1$, which is isomorphic to $F$ and satisfies the following conditions:
	$$V(L_1)=\{1,2,...,n-2\} \cup \{l,2l\}, \, 
	deg_{L_1} (2l)=t,$$
	$$(1,2l), (2,2l),..., (t,2l) \in E(L_1), \, (1,l) \in E(L_1) .$$
	The existence of subgraph $L_1$ follows from the facts that $(1,2l), (2,2l),..., (t,2l) \in E(F_{2l})$, $\delta (F)=t$, the graph $F$ is connected, $t \leqslant n-2$, and any distinct vertices from the set $\{1,2,...,n-2\}\cup \{l\}$ are adjacent in the graph $F_{2l}$.
	
	Now we prove that $\delta_l <z_{i+1}-z_i$.
	
	Based on Lemmas \ref{lem7} and \ref{lem2}, we have the following estimates for $\delta_l$ and $C_{l-2}^{n-2}$:
	$$\delta_l \leqslant n!C_{l-2}^{n-t-2}, \quad
	z_{i+1}- z_i \geqslant C_{l-2}^{n-2} \quad \forall i \in \{t+1,t+2,...,l-1\}.$$
	Next, from the assumption  $C_{l-2}^{n-2}>n!C_{l-2}^{n-t-2}$
	of Lemma \ref{lem10}, we obtain that the inequality $\delta_l <z_{i+1}-z_i$ holds for each $i\in \{t+1,t+2,...,l-1\}$.
\end{proof}

\begin{lemma} \label{lem11}  
	Let $t \geqslant 2$, $i\in \{l+1,l+2,...,l+t-1\}$. The jumps of vertices $i$ and $l$ in the graph $A_{2l-1}$ satisfy the inequality $$\delta_i < \delta_l.$$			
\end{lemma}
\begin{proof} 
	Let us construct a lower bound for the difference $\delta_l - \delta_i$. Lemma \ref{lem4} and Assertion \ref{as4} imply that $\delta_i$ and $\delta_l$ are equal to the number of subgraphs of $X$ that are isomorphic to $F$ and contain vertices $i,2l$ and $l,2l$, respectively.
	
	Next, consider the graph $X^-$, which is formed from $X$ by deleting edges $(1,l),(2,l),...,$ $(i-l,l)$. Note that in the graph \(X^-\), the closed neighborhoods of vertices $i$ and $l$ coincide: $$N_{X^-}[i]=N_{X^-}[l] =\{i-l+1,i-l+2,...,l+t-1\}.$$
	Therefore, there exists an automorphism $p:V(X^-)\rightarrow V(X^-)$ such that
	$$p(i)=l, p(l)=i, p(v)=v \quad \forall v \in V(X^-) \backslash \{i, l\}.$$
	Consequently, the vertices $i$ and $l$ belong to the same number of subgraphs of $X^-$ that are isomorphic to $F$ and contain the vertex $2l$. Then the difference $\delta_l -\delta_i$ will be equal to
	$$\delta_l -\delta_i=|K|,$$
	where $K$ is the set of all subgraphs of $X$ that are isomorphic to $F$, contain vertices $l, 2l$, do not contain vertex $i$, and contain at least one edge from the set $\{(1,l),(2,l),...,(i-l,l)\}$. It remains to note that the set $K$ is not empty, since the subgraph $L_1$ from the proof of Lemma \ref{lem10} belongs to $K$. Consequently,
	$\delta_l -\delta_i>0$, which implies the inequality
	$\delta_i < \delta_l $ for each $i\in \{l+1,l+2,...,l+t-1\}$.
\end{proof}	
	\begin{center}
		\section{Proof of the Theorem \ref{th4}}  \label{sec3} 
	\end{center}
	
	\subsection{$\langle n,t \rangle$-condition} 
	
	\begin{definition}\label{def11}  
		
		Let $\langle n,t \rangle$ be an ordered pair of positive integers $n$ and $t$ such that $t \leqslant n-2$. We say that a positive integer $l$ satisfies the $\langle n,t \rangle$-condition if
		\begin{center}
			$ l>(n!+1)(n-1)$ for $t=1$;
			
			\
			
			$l>2t$, $C_{l-2}^{n-2}>n!C_{l-1}^{n-t-1}$ for $t \geqslant 2$.				
		\end{center}	
	\end{definition}
	
	\begin{lemma} \label{lem12}  
		For every ordered pair of positive integers $\langle n,t \rangle$ such that $t \leqslant n-2$, there exist infinitely many positive integers $l$ for which the $\langle n,t \rangle$-condition holds.	
	\end{lemma}
	
	\begin{proof} 
	Fix an ordered pair of positive integers $\langle n,t \rangle$  with $t \leqslant n-2$.
		
	If $t=1$, the statement of Lemma \ref{lem12} is evident.
		
	Now suppose \(t\geqslant 2\). For fixed  positive integers $n,t$ and for arbitrary positive integer $l>n \geqslant t+2$, we have equivalence 
	$$C_{l-2}^{n-2}>n!C_{l-1}^{n-t-1} \iff
	\frac{(l-2)!}{(n-2)!(l-n)!}> n!
	\frac{(l-1)!}{(n-t-1)!(l-n+t)!} \iff$$
	$$(l-n+1)(l-n+2)...(l-n+t)-\frac{n!(n-2)!}{(n-t-1)!}(l-1)>0.$$
	
	Next, consider the polynomial $P(x)$ of degree $t \geqslant 2$ defined by 		
	$$P(x)=(x-n+1)(x-n+2)...(x-n+t)-\frac{n!(n-2)!}{(n-t-1)!}(x-1).$$  			
	Since the leading coefficient of $P(x)$ is positive (equal to $1$), there exists a  real number $x_0$ such that for all $x>x_0$ the polynomial $P(x)$ takes only positive values. In particular, for every  positive integer $l>x_0$,  the inequality 
	$$P(l)=(l-n+1)(l-n+2)...(l-n+t)-\frac{n!(n-2)!}{(n-t-1)!}(l-1)>0.$$
	holds. Hence, for each positive integer $l$ such that $l>x_0$ and $l>n$, the inequality 
	$$C_{l-2}^{n-2}>n!C_{l-1}^{n-t-1}$$
	is satisfied. It follows that any positive integer $l$ with $l>max(x_0, n, 2t)$  
	satisfies the $\langle n,t \rangle$-condition.
	To complete the proof, it remains to note that there are infinitely many positive integers  $l$ of this type. 	
	\end{proof}
	
	\begin{lemma} \label{lem13}  
		Let $\langle n,t \rangle$ be an ordered pair of positive integers $n$ and $t$ such that $t \leqslant n-2$. If a positive integer $l$ satisfies the $\langle n,t \rangle$-condition, then the following inequalities hold:
		\vspace{5pt}
		
		$ 1) \: l>n $;
		
		\vspace{5pt}
		
		$ 2) \: C_{l-1}^{n-1}>n!C_{l-1}^{n-t-1}$;
		
		\vspace{5pt}
		
		$ 3) \: C_{l-2}^{n-2}>n!C_{l-2}^{n-t-2}$.
		
	\end{lemma}
	
	\begin{proof} 
		Suppose that $l$ satisfies the $\langle n,t \rangle$-condition. From the definition of the pair $\langle n,t \rangle$, it follows that $n\geqslant 3$.
		
		We prove $1)$. For $t=1$, $n\geqslant 3$,  the $\langle n,1 \rangle$-condition implies $$ l>(n!+1)(n-1)>n.$$
		If $t \geqslant 2$, then by the $\langle n,t \rangle$-condition we have $C_{l-2}^{n-2}>n!C_{l-1}^{n-t-1},$ which implies $C_{l-2}^{n-2}>0$,  
		\vspace{4pt}
		
		\hspace{-21pt} and  thus  $l\geqslant n$. 
		If $l=n$, then the inequality $C_{l-2}^{n-2}>n!C_{l-1}^{n-t-1}$ becomes
		$1>n!C_{n-1}^{n-t-1}$,  
		\vspace{4pt}
		
		\hspace{-21pt} which is false, since $C_{n-1}^{n-t-1}>1$ for all positive integers $n$ and $t \leqslant n-2$. Therefore, $l>n$ for $t \geqslant 2$.	
			
		Due to the validity of inequality $1)$, in the following proofs of $2)$-$3)$, we  assume that $l,n,t$ are positive integers such that $l>n\geqslant t+2$.
		\vspace{5pt}
		
		Let us prove $2)$. If $t=1$, then the inequality $C_{l-1}^{n-1}> n!C_{l-1}^{n-2}$ holds, since it is equivalent to the inequality from the $\langle n,1 \rangle$-condition:
		$$C_{l-1}^{n-1}> n!C_{l-1}^{n-2} \iff \frac{(l-1)!}{(n-1)!(l-n)!}> n!
		\frac{(l-1)!}{(n-2)!(l-n+1)!} \iff$$
		$$l>(n!+1)(n-1).$$
		
		If $t \geqslant 2$, then $C_{l-2}^{n-2}>n!C_{l-1}^{n-t-1}$ in accordance with the $\langle n,t \rangle$-condition, and to prove 
		\vspace{4pt}
		
		\hspace{-21pt}
		the inequality $C_{l-1}^{n-1}>n!C_{l-1}^{n-t-1}$, it suffices to show that
		$$C_{l-1}^{n-1} \geqslant C_{l-2}^{n-2}.$$
		This inequality holds because it is equivalent to the true statement:
		$$C_{l-1}^{n-1} \geqslant C_{l-2}^{n-2} \iff \frac{(l-1)!}{(n-1)!(l-n)!}\geqslant
		\frac{(l-2)!}{(n-2)!(l-n)!}
		\iff l \geqslant n.$$
		\vspace{1pt}
		
		We prove $3)$. For $t=1$, the inequality $C_{l-2}^{n-2}> n!C_{l-2}^{n-3}$ is equivalent to the following:
		$$C_{l-2}^{n-2}> n!C_{l-2}^{n-3} \iff \frac{(l-2)!}{(n-2)!(l-n)!}> n!
		\frac{(l-2)!}{(n-3)!(l-n+1)!} \iff$$ 
		$$l>n!(n-2)+n-1.$$		
		Next, from the  $\langle n,1 \rangle$-condition, we have
		$$ l>(n!+1)(n-1)=n!(n-1)+n-1>n!(n-2)+n-1,$$
		which implies $l>n!(n-2)+n-1$, and hence  $C_{l-2}^{n-2}> n!C_{l-2}^{n-3}$.
		\vspace{3pt}
		
		If $t \geqslant 2$, the desired  inequality follows from the $\langle n,t \rangle$-condition and the auxiliary inequality
		$$C_{l-1}^{n-t-1} \geqslant C_{l-2}^{n-t-2},$$
		whose truth is established by the following equivalent transformations:
		$$C_{l-1}^{n-t-1} \geqslant C_{l-2}^{n-t-2} \iff \frac{(l-1)!}{(n-t-1)!(l-n+t)!}\geqslant
		\frac{(l-2)!}{(n-t-2)!(l-n+t)!} \iff$$
		$$l \geqslant n-t.$$	
	\end{proof}
	
	\subsection{Proof of the Theorem \ref{th4}} 
	
	\
			
	Consider an arbitrary graph $F$ of diameter $2$. Let $n=|F|$, $t=\delta (F)$. From Assertion \ref{as3} it follows that for positive integers $n,t$ the inequality $t \leqslant n-2$ holds. Choose a positive integer $l$ that satisfies the $\langle n,t \rangle$-condition. Then, by Lemma \ref{lem13}, we have $l>n$, hence for this chosen  $l$ there exists a graph $F_{2l}$ on $2l$ vertices, whose description is given  in section \ref{sec2}. Additionally, note that from the $\langle n,t \rangle$-condition it follows that $l>2t$ for  every positive integer $t$.
	Furthermore, by Lemma \ref{lem3}, the diameter of the graph $F_{2l}$ is equal to $3$.
	\vspace{3pt}
	
	Now we  prove that $F_{2l}$ is $F$-irregular. To compare the $F$-degrees of vertices in $F_{2l}$, we turn to the following visualization idea for  the transformation of the graph $A_{2l-1}$ into the graph $F_{2l}$.	
	\vspace{3pt}
		
	Consider the graph $A_{2l-1}$. Color the vertices  $1,2,...,t$ of this graph red, the vertices $t+1,t+2,...,l$ blue, the vertices $l+t,l+t+1,...,(2l-1)$ black, and the vertices $l+1,l+2,...,(l+t-1)$ green (when $t \geqslant 2$).
	\vspace{3pt}
		
	Based on Lemma \ref{lem1} and Corollary \ref{cor1}, all vertices of $A_{2l-1}$ can be placed on exactly $l$ floors such that the height of each floor numerically coincides with the $F$-degree of every vertex located on that floor. Moreover, if we number the floors  $1, 2,..., l$ from bottom to top, then for every $i \in \{1,2,...,l-1\}$, the vertices $i$ and $2l-i$ lie on the $i$-th floor, and the vertex $l$ lies on the $l$-th floor. By construction, for any $i \in \{1,2,...,l-1\}$, the distance between floors $i+1$ and $i$ equals $z_{i+1}-z_i$.
	\vspace{3pt}
	
	Next, we  analyze how the $F$-degrees of vertices in the graph $A_{2l-1}$ change when transitioning to the graph $F_{2l}$ (see Figure  4). 
	
	\begin{figure}[h!]
		\centering
		\includegraphics[width=13.5cm]{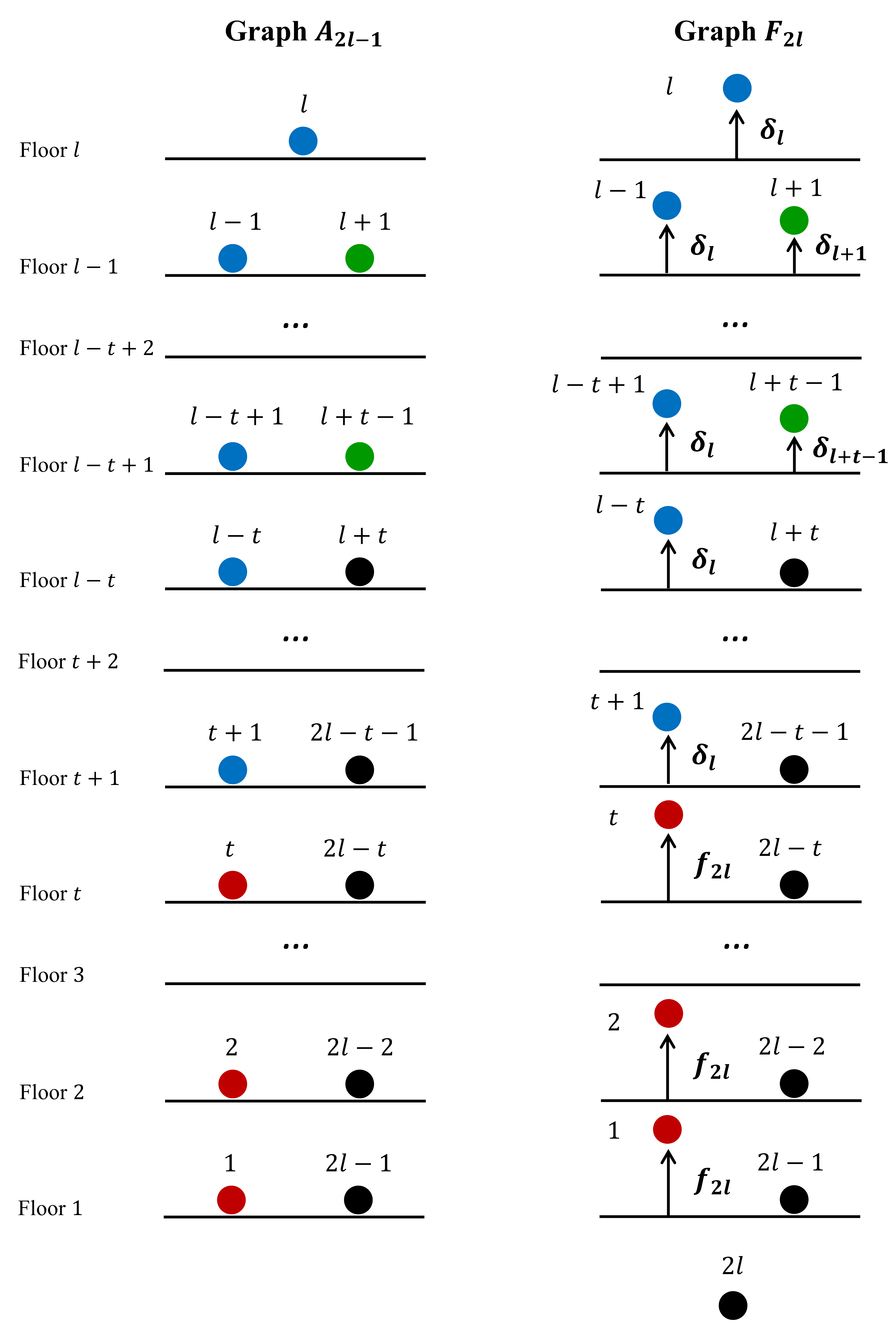}
		\caption{Diagram illustrating the evolution of $F$-degrees
			of vertices in graph $A_{2l-1}$ when transitioning to graph $F_{2l}$.}
		\label{fig4}
	\end{figure}		
	\vspace{3pt}
		
	Consider the floor numbered $i \in \{1,2,...,t\}$. Since
	$l>2t$, on this floor there is the red vertex $i \leqslant t$ and the black vertex $2l-i > l+t$. By Lemmas \ref{lem5} and \ref{lem9}, in the transition from  $A_{2l-1}$ to  $F _{2l}$ the red vertex $i$  "jumps" up by height  $\delta_i=f_{2l}>0$, but does not reach the next floor $(f_{2l}<z_{i+1}-z_i)$, whereas the black vertex $2l-i$ remains fixed ($\delta_{2l-i}=0$).
    \vspace{3pt}
    
	Let $i \in \{t+1,t+2,...,l-t\}$. Since $l>2t$, on the $i$-th floor there is the blue vertex $i$ with $t+1 \leqslant i \leq l-t$, and the black vertex $2l-i \geqslant l+t$. According to Lemmas \ref{lem5}, \ref{lem10}, \ref{lem13}, as a result of transforming $A_{2l-1}$ into $F_{2l}$, the blue vertex $i$  "jumps" by $\delta_i=\delta_l>0$, but does not reach the next floor, since $\delta_l<z_{i+1}-z_i$, while the black vertex $2l-i$ on the same floor does not change its position ($\delta_{2l-i}=0$).
	\vspace{3pt}
	
	If $t \geqslant 2$ and $i \in \{l-t+1,l-t+2,...,l-1\}$, then, taking into account the inequality $ l>2t$, on the $i$-th floor there are vertices $i$ and $2l-i$ of blue and green colors, respectively. 
	When transitioning to $F_{2l}$, based on Lemmas \ref{lem5}, \ref{lem10}, \ref{lem11}, \ref{lem13}, the blue vertex $i$ "jumps" by $\delta_i=\delta_l>0$, but does not reach the next floor, while the jump of the corresponding green vertex $2l-i$ satisfies the inequality $\delta_{2l-i}<\delta_l=\delta_i$, hence in   $F _{2l}$ the green vertex $2l-i$ is located below the blue vertex $i$. Moreover, from Assertion \ref{as4} it follows that $\delta_{2l-i} \geqslant 0$, so the green vertex $2l-i$ in  $F _{2l}$ lies no lower than the $i$-th floor.
	\vspace{3pt}
	
	Finally, by Lemma \ref{lem10}, the vertex $l$ in  $F_{2l}$ is located above the floor numbered $l$ ($\delta_l>0$), and by Lemmas \ref{lem8}, \ref{lem13}, the vertex $2l$ must be placed below the first floor ($f_{2l}<z_1$).
	\vspace{3pt}
					
	As a result, all vertices of  $F_{2l}$  end up at distinct heights. Consequently, the $F$-degrees of  vertices in  $F_{2l}$ are pairwise distinct,  hence $F_{2l}$ is an $F$-irregular graph.
	\vspace{3pt}
		
	To complete the proof, it remains to note that, by Lemma \ref{lem12}, there exist infinitely many positive integer values of $l$ for which the $\langle n,t \rangle$-condition holds.
	Thus, the set of all graphs $\{F_{2l}\}$, where the positive integer $l$ satisfies the $\langle n,t \rangle$-condition, forms an infinite series of $F$-irregular graphs of diameter $3$. 
	
	Theorem \ref{th4}  is proved.
	
	\begin{remark} \label{rem1} Note that in proving the $F$-irregularity of graph $F_{2l}$, we operate with only three characteristics of graph $F$: its order, diameter, and the minimum vertex degree. Therefore, if a positive  integer $l$ satisfies the $\langle n,t \rangle$-condition, where $n=|F|$, $t=\delta(F)$, and $F$ is a graph of diameter $2$, then $F_{2l}$ is $H$-irregular for any graph $H$ of diameter $2$ whose order is $|H|=|F|=n$ and  minimum vertex degree $\delta (H)=\delta (F)=t$. 
	\end{remark}

	\begin{center}
		\section{Conclusion} \label{sec4} 
	\end{center}
	
	In this article, for every graph $F$ of diameter $2$, we prove the existence of infinitely many $F$-irregular graphs of diameter $3$. Thus, the strong conjecture about  $F$-irregular graphs is confirmed within the class of graphs $\{F\}$ of diameter $2$.

	We hope that the proposed constructions and techniques for comparing $F$-degrees of vertices will be useful for finding a general approach to proving the strong conjecture about $F$-irregular graphs.
	
	Furthermore, as noted in Remark \ref{rem1}, there exist graphs that possess the $F$-irregularity property for different graphs $F$. On this basis, we propose to extend the concept of $F$-irregular graphs by considering a new class of graphs --  $n$-\emph{hyper-irregular graphs}. \\
		
	\begin{definition}\label{def12}  
		Let $n \geqslant 3$ be an integer. A graph $G$ is called $n$-hyper-irregular if it is $F$-irregular for every graph $F$ (not necessarily connected)  with  order $|F| \leqslant n$ and size $|E(F)| \geqslant 2$.
	\end{definition}
	
	We also put forward a new conjecture about $F$-irregular graphs, stronger than Conjecture \ref{con2}.	
		
	\begin{conjecture} \label{con3} \emph{(Super-strong conjecture on  $F$-irregular graphs)}
		For each integer $n\geqslant 3$, there exist infinitely many $n$-hyper-irregular graphs.
	\end{conjecture}

\end{document}